\documentclass{amsart}
\usepackage[a4paper,centering,textwidth=145mm,textheight=195mm]{geometry}
\def\C{\mathbb{C}}
\def\Q{\mathbb{Q}}

\DeclareMathOperator{\charac}{char}

\newcommand{\qf}[1]{{\langle{#1}\rangle}} 
\newcommand{\pf}[1]{{\langle\!\langle{#1}\rangle\!\rangle}} 

\newtheorem*{thmA}{Theorem A}
\newtheorem*{ques}{Question}
\newtheorem*{thmB}{Theorem B}

\newtheorem{prop}{Proposition}

\theoremstyle{definition}
\newtheorem{notation}[prop]{Notation}

\title{Linkage of Pfister forms over $\C(x_1,\ldots,x_n)$}
\author{Adam Chapman}
\address{Department of Computer Science, Tel-Hai Academic College,
  Upper Galilee, 12208 Israel} 
\email{adam1chapman@yahoo.com}

\author{Jean-Pierre Tignol}
\address{ICTEAM Institute, 
UCLouvain, Box L4.05.01,
B-1348 Louvain-la-Neuve, Belgium}
\email{jean-pierre.tignol@uclouvain.be}
\thanks{The second author acknowledges support
    from the Fonds de la Recherche Scientifique--FNRS under grant
    n$^\circ$~J.0159.19.}
  \date{\today}

  
\begin{document}
\maketitle
\begin{abstract}
In this note, we prove the existence of a set of $n$-fold Pfister
forms of cardinality $2^n$ over $\C(x_1,\dots,x_n)$ which do
not share a common $(n-1)$-fold factor. This gives a negative answer
to a question raised by Becher. The main tools are the existence of
the dyadic valuation on the complex numbers and recent results on
symmetric bilinear over fields of characteristic 2.
\end{abstract}

The field $\C(x_1,x_2)$ of rational functions in two indeterminates
over the field of complex numbers is known to be a $C_2$-field in the
sense of Lang (see \cite[Section~97]{EKM}). It follows that every
quadratic form in five variables over $\C(x_1,x_2)$ is isotropic,
which implies that any two quaternion algebras over $\C(x_1,x_2)$
share a common maximal subfield, see
\cite[Th.~X.4.20]{Lam:2005}. Fields 
with this property are said to be \emph{linked}. It was noticed by
Becher in \cite{Becher} and by Chapman--Dolphin--Leep in
\cite[Cor.~5.3]{ChapmanDolphinLeep} that the following stronger
property holds: 
$\C(x_1,x_2)$ is \emph{$3$-linked} in the sense that any \emph{three}
quaternion algebras over $\C(x_1,x_2)$ share a common maximal
subfield. Comparison with the case of number fields, which are
$m$-linked for every integer $m$ by the local-global principle (see
\cite[Ex.~X.5.12A]{Lam:2005}), suggests to ask whether there exists an upper
bound on the integer $m$ for which $\C(x_1,x_2)$ is $m$-linked. We prove below:

\begin{thmA}
  The following quaternion algebras over $\C(x_1,x_2)$ do not share a
  common maximal subfield:
  \[
    (x_1,x_2),\qquad (x_1,x_2+1),\qquad (x_2,x_1+1),\qquad
    (x_2,x_1x_2+1).
  \]
\end{thmA}

The arguments apply to a more general linkage question raised by
Becher~\cite{Becher}. Given a field $F$, the Witt ring $W F$ of (Witt classes of) symmetric bilinear forms over $F$ has a natural filtration by the powers of the maximal ideal $I F$ of even-dimensional forms:
$$W F \supset I F \supset I^2 F \supset \dots $$
Each $I^n F$ is generated by (bilinear) $n$-fold Pfister forms, i.e.,
forms of the shape
\[
  \pf{\alpha_1,\ldots,\alpha_n} = \qf{1,-\alpha_1}\otimes \cdots
  \otimes \qf{1,-\alpha_n}.
\]
For $m$, $n\geq2$, we say that $I^n F$ is \emph{$m$-linked} if every
$m$ bilinear $n$-fold Pfister forms over $F$ share a common
$(n-1)$-fold factor. If $\charac(F)\neq2$, quadratic forms can be
identified with their symmetric bilinear polar forms, and in
particular the $2$-fold Pfister forms are the norm forms of quaternion
algebras, hence $F$ is $m$-linked in the sense discussed above if and
only if $I^2F$ is $m$-linked. Becher raised the following question:

\begin{ques}[{\cite[Question 5.2]{Becher}}]
  \label{Becher}
  Suppose $I^n F$ is $3$-linked for some $n\geq2$. Does it follow that
  $I^n F$ is $m$-linked for every $m \geq 3$?
\end{ques}

This question was answered in the negative for fields $F$ of
$\charac(F)=2$ in \cite{Chapman:2018}. 
In this note, we shall show how Becher's question can be answered also
in the case of $\operatorname{char}(F)=0$ using the main result of
\cite{Chapman:2018} on symmetric bilinear forms over fields of
characteristic~$2$ and the existence of a dyadic valuation on $\C$:

\begin{thmB}
  For $F=\C(x_1,\ldots,x_n)$ with $n\geq2$, $I^nF$ is $3$-linked but
  not $2^n$-linked.
\end{thmB}

\section*{Proofs}

\begin{notation}
  \label{not:1}
  For a given integer $n\geq2$, let $\bold{2}^n=\{0,1\}^{\times n}$,
  and write 
  $\bold{0}=(0,\dots,0)\in \bold{2}^n$. Given a sequence $\alpha_1$,
  \ldots, 
  $\alpha_n$ in the multiplicative group of a field $F$ and
  $\bold{d}=(d_1,\ldots,d_n)\in \bold{2}^n$, let $\alpha^{\bold 
    d}=\prod_{i=1}^n\alpha^{d_i}\in F^\times$. If
  $\bold{d}\neq\bold{0}$, let
  \[
    \varphi_{\bold d} = \pf{\alpha_1, \ldots, \widehat{\alpha_\ell},
      \ldots, \alpha_n}\otimes\pf{1+\alpha^{\bold d}},
  \]
  where $\ell$ is the minimal index in $\{1, \ldots, n\}$ for which
  $d_\ell\neq0$, and let
  \[
    \varphi_{\bold0} = \pf{\alpha_1, \ldots, \alpha_n}.
  \]
\end{notation}

The following result is from~\cite[Th.~3.3]{Chapman:2018}:

\begin{prop}
  \label{prop:char2}
  Suppose $\charac(F)=2$ and $\alpha_1$, \ldots, $\alpha_n$ are
  $2$-independent in $F$, which means that
  $(\alpha^{\bold d})_{\bold{d}\in \bold{2}^n}$ is a linearly
  independent family in $F$ viewed as an
  $F^2$-vector space. Then the forms $\varphi_{\bold d}$ for
  $\bold{d}\in \bold{2}^n$ are anisotropic and have no common $1$-fold
  factor.
\end{prop}

The main result from which Theorems~A and B derive is the following:

\begin{prop}
  \label{prop:main}
  Let $F=k(x_1, \ldots, x_n)$ be the field of rational functions in
  $n$ indeterminates over an arbitrary field $k$ of characteristic
  zero, for some $n\geq2$. Let $\varphi_{\bold d}$ for
  $\bold{d}\in\bold{2}^n$ be the Pfister forms defined as in
  Notation~\ref{not:1} with the sequence $x_1$, \ldots, $x_n$ for
  $\alpha_1$, \ldots, $\alpha_n$. The
  forms $\varphi_{\bold d}$ do not have a common $1$-fold factor.
\end{prop}

\begin{proof}
  A theorem of Chevalley (see \cite[Theorem~3.1.1]{EnglerPrestel})
  shows that the $2$-adic valuation on $\Q$ extends to a valuation
  $v_0$ on $k$. Let $\overline k$ be the residue field of this
  valuation, which has characteristic~$2$. The valuation $v_0$ has a
  Gauss extension to a valuation $v$ on $F$ such that $v(x_i)=0$ for
  $i=1$, \ldots, $n$ and $\overline{x_1}$, \ldots, $\overline{x_n}$
  are algebraically 
  independent over $\overline k$; see
  \cite[Cor.~2.2.2]{EnglerPrestel}. The residue field of $v$ is thus 
  $\overline{F}=\overline{k}(\overline{x_1}, \ldots, \overline{x_n})$,
  a field of rational functions in $n$ indeterminates over $\overline
  k$. Since the coefficients of the forms $\{\varphi_{\bold{d}} :
  \bold{d} \in \bold{2}^n\}$ are all of value $0$, they have residue
  forms $\{\overline{\varphi}_{\bold{d}} : \bold{d} \in \bold{2}^n\}$,
  where the coefficients of $\overline{\varphi}_{\bold{d}}$ are the
  residues of the coefficients of $\varphi_{\bold{d}}$. The forms
  $\overline{\varphi}_{\bold d}$ are bilinear Pfister forms as defined
  in Notation~\ref{not:1}, with the $2$-independent sequence
  $\overline{x_1}$, \ldots, $\overline{x_n}$ for $\alpha_1$, \ldots,
  $\alpha_n$. 

  For $\bold{d}\in\bold{2}^n$, let $\bold{t}_{\bold d} = (t_{1,\bold
    d}, \ldots, t_{2^n-1,\bold d})$ be a $(2^n-1)$-tuple of
  indeterminates. Suppose the bilinear forms $\varphi_{\bold d}$ have
  a common factor $\pf{\alpha}$. Then the pure subforms
  $\varphi'_{\bold d}$ defined by the equation $\varphi_{\bold d} =
  \qf 1 \perp \varphi'_{\bold d}$ all represent $-\alpha$, hence the
  system of equations
  \[
    \varphi'_{\bold d}(\bold{t}_{\bold d}, \bold{t}_{\bold d}) =
    -\alpha \qquad\text{for $\bold{d}\in\bold{2}^n$}
  \]
  has a solution. We may therefore find nontrivial solutions to the
  system of equations
  \[
    \varphi'_{\bold d}(\bold{t}_{\bold d}, \bold{t}_{\bold d}) =
    \varphi'_{\bold 0}(\bold{t}_{\bold 0}, \bold{t}_{\bold0})
    \qquad\text{for $\bold{d}\in\bold{2}^n\setminus\{\bold{0}\}$.}
  \]
  Since these equations are homogeneous, upon scaling we may find
  solutions $(\bold{u}_{\bold d})_{\bold{d}\in\bold{2}^n}$ such that
  \[
    \min\{v(u_{i,\bold d})\mid i=1, \ldots, 2^n-1,\;
    \bold{d}\in\bold{2}^n\} = 0.
  \]
  Taking residues, we obtain
  \[
    \overline{\varphi}'_{\bold d}(\overline{\bold{u}_{\bold d}},
    \overline{\bold{u}_{\bold d}}) = \overline{\varphi}'_{\bold 0}
    (\overline{\bold{u}_{\bold0}}, \overline{\bold{u}_{\bold0}})
    \qquad\text{for $\bold{d}\in\bold{2}^n\setminus\{\bold{0}\}$.}
  \]
  Since at least one $\overline{u_{i,\bold d}}$ is nonzero and the
  forms $\overline{\varphi}'_{\bold d}$ are anisotropic, it follows
  that these forms all represent some $\beta\in\overline{F}^\times$,
  hence the forms $\overline{\varphi}_{\bold d}$ have a common factor
  $\pf{\beta}$ by \cite[Lemma~6.11]{EKM}. This yields a contradiction
  to Proposition~\ref{prop:char2}.
\end{proof}

Theorem~A readily follows from Proposition~\ref{prop:main} with $n=2$
and $k=\C$, because the forms $\varphi_{\bold0}$, $\varphi_{(0,1)}$,
$\varphi_{(1,0)}$, and $\varphi_{(1,1)}$ are the norm forms of the
quaternion algebras $(x_1,x_2)$, $(x_1, x_2+1)$, $(x_2,x_1+1)$ and
$(x_2, x_1x_2+1)$ respectively.

\begin{proof}[Proof of Theorem~B]
  The field $F=\C(x_1, \ldots, x_n)$ is a $C_n$-field, hence $F(t)$ is
  a $C_{n+1}$-field, see \cite[Cor.~97.6]{EKM}. In particular,
  $u\bigl(F(t)\bigr)=2^{n+1}$, and it follows
  from~\cite[Cor.~5.4]{Becher} that $I^nF$ is $3$-linked. Apply
  Proposition~\ref{prop:main} with $k=\C$ to obtain a set of $n$-fold
  Pfister forms of cardinality $2^n$ that do not have a common
  $1$-fold factor, hence are not linked.
\end{proof}

\bibliographystyle{plain}
\bibliography{bibfile}

\end{document}